\documentclass[12pt,a4paper]{amsart}

\usepackage[utf8]{inputenc}
\usepackage{enumitem,a4,titletoc}
\usepackage{bbm,mathtools,xspace}
\usepackage{MnSymbol}
\usepackage[usenames,dvipsnames]{xcolor}
\usepackage[pagebackref,colorlinks,linkcolor=BrickRed,citecolor=OliveGreen,urlcolor=black,hypertexnames=true]{hyperref}

\usepackage[margin=2.5cm]{geometry}

\DeclareMathOperator{\diag}{diag}

\DeclareMathOperator{\trc}{tr}
\DeclareMathOperator{\rk}{rk}
\DeclareMathOperator{\End}{End}
\DeclareMathOperator{\im}{im}
\DeclareMathOperator{\spa}{span}
\DeclareMathOperator{\GL}{GL}

\DeclareMathOperator{\dom}{dom}
\DeclareMathOperator{\domh}{dom_{sa}}
\DeclareMathOperator{\real}{Re}
\DeclareMathOperator{\imag}{Im}
\DeclareMathOperator{\opm}{M}
\newcommand{\ve}{\varepsilon}
\newcommand{\N}{\mathbb{N}}
\newcommand{\R}{\mathbb{R}}
\newcommand{\C}{\mathbb{C}}
\newcommand{\bA}{\mathbf{A}}
\newcommand{\bC}{\mathbf{C}}

\newcommand{\cC}{\mathcal{C}}

\newcommand{\cK}{\mathcal{K}}

\newcommand{\cR}{\mathcal{R}}

\newcommand{\cS}{\mathcal{S}}

\newcommand{\cU}{\mathcal{U}}
\newcommand{\cV}{\mathcal{V}}

\newcommand{\bb}{\mathbf{b}}
\newcommand{\cc}{\mathbf{c}}
\newcommand{\uu}{\mathbf{u}}
\newcommand{\vv}{\mathbf{v}}
\newcommand{\mm}{\mathbf{m}}
\newcommand{\ww}{\mathbf{w}}

\newcommand{\rr}{\mathbbm r}
\newcommand{\qq}{\mathbbm q}
\newcommand{\rs}{\mathbbm s}

\newcommand{\kk}{\mathbbm k}
\newcommand{\mpe}{^{\operatorname{mp}}}

\newcommand{\ulx}{\boldsymbol{x}}
\newcommand{\uly}{\boldsymbol{y}}

\newcommand*{\mat}[1]{\opm_{#1}(\kk)}
\newcommand*{\herm}[1]{\opm_{#1}(\kk)_{\operatorname{sa}}}
\newcommand{\all}{\mathcal{M}^g}
\newcommand{\allh}{\mathcal{M}^g_{\operatorname{sa}}}

\newcommand{\Langle}{\mathop{<}\!}
\newcommand{\Rangle}{\!\mathop{>}}
\newcommand{\mx}{\Langle \ulx\Rangle}
\newcommand{\px}{\kk\!\Langle \ulx\Rangle}

\makeatletter
\def\moverlay{\mathpalette\mov@rlay}
\def\mov@rlay#1#2{\leavevmode\vtop{
		\baselineskip\z@skip \lineskiplimit-\maxdimen
		\ialign{\hfil$#1##$\hfil\cr#2\crcr}}}
\makeatother

\newcommand{\re}{\cR_{\kk}(\ulx)}
\newcommand{\plangle}{\moverlay{(\cr<}}
\newcommand{\prangle}{\moverlay{)\cr>}}
\newcommand{\rx}{\kk\plangle \ulx \prangle}

\newcommand{\rxxc}{\C\plangle \ulx,\ulx^* \prangle}
\newcommand{\rxyc}{\C\plangle \ulx,\uly \prangle}

\def\cVsa{\cV^{\operatorname{sa}}}

\newcommand{\privileged}{free elliptic\xspace}
\newcommand{\Privileged}{Free elliptic\xspace}
\newcommand{\stably}{strongly\xspace}

\newtheorem{thm}{Theorem}[section]
\newtheorem{lem}[thm]{Lemma}
\newtheorem{cor}[thm]{Corollary}
\newtheorem{prop}[thm]{Proposition}
\newtheorem{thmA}{Theorem}

\theoremstyle{definition}
\newtheorem{defn}[thm]{Definition}
\newtheorem{exa}[thm]{Example}

\theoremstyle{remark}
\newtheorem{rem}[thm]{Remark}

\numberwithin{equation}{section}

\begin{document}
	
\setcounter{tocdepth}{3}
\contentsmargin{2.55em} 
\dottedcontents{section}[3.8em]{}{2.3em}{.4pc} 
\dottedcontents{subsection}[6.1em]{}{3.2em}{.4pc}
\dottedcontents{subsubsection}[8.4em]{}{4.1em}{.4pc}

\makeatletter
\newcommand{\mycontentsbox}{%

	{\centerline{NOT FOR PUBLICATION}
		\addtolength{\parskip}{-2.3pt}
		\tableofcontents}}
\def\enddoc@text{\ifx\@empty\@translators \else\@settranslators\fi
	\ifx\@empty\addresses \else\@setaddresses\fi
	\newpage\mycontentsbox\newpage\printindex}
\makeatother

\setcounter{page}{1}

\title[Regular and positive noncommutative rational functions]{Regular and positive\\[.2mm] noncommutative rational functions}

\author[I. Klep]{Igor Klep${}^1$}
\address{Igor Klep, Department of Mathematics, University of Auckland}
\email{igor.klep@auckland.ac.nz}
\thanks{${}^1$Supported by the Marsden Fund Council of the Royal Society of New Zealand. Partially supported by the Slovenian Research Agency grants P1-0222 and L1-6722.}

\author[J. E. Pascoe]{James Eldred Pascoe${}^2$}
\address{James Eldred Pascoe, Department of Mathematics, Washington University in St. Louis}
\email{pascoej@wustl.edu}
\thanks{${}^2$Supported by a National Science Foundation Mathematical Sciences Postdoctoral Research Fellowship Award No. DMS-1606260.}

\author[J. Vol\v{c}i\v{c}]{Jurij Vol\v{c}i\v{c}${}^3$}
\address{Jurij Vol\v{c}i\v{c}, Department of Mathematics, University of Auckland}
\email{jurij.volcic@auckland.ac.nz}
\thanks{${}^3$Supported by University of Auckland Doctoral Scholarship.}

\date{\today}

\begin{abstract}
Call a noncommutative rational function $\rr$ regular if it has no singularities, i.e., $\rr(X)$ is defined for all tuples of self-adjoint matrices $X$. In this article regular noncommutative rational functions $\rr$ are characterized via the properties of their (minimal size) linear systems realizations $\rr=\bb^* L^{-1}\cc$. It is shown that $\rr$ is regular if and only if $L=A_0+\sum_jA_j x_j$ is \privileged. Roughly speaking, a linear pencil $L$ is \privileged if, after a finite sequence of basis changes and restrictions, the real part of $A_0$ is positive definite and the other $A_j$ are skew-adjoint. The second main result is a solution to a noncommutative version of Hilbert's 17th problem: a positive regular noncommutative rational function is a sum of squares.
\end{abstract}

\iffalse
%%%%%%%% TEXT-only abstract 
Call a noncommutative rational function $r$ regular if it has no singularities, i.e., $r(X)$ is defined for all tuples of self-adjoint matrices $X$. In this article regular noncommutative rational functions $r$ are characterized via the properties of their (minimal size) linear systems realizations $r=b^* L^{-1} c$. It is shown that $r$ is regular if and only if $L=A_0+\sum_jA_j x_j$ is \privileged. Roughly speaking, a linear pencil $L$ is \privileged if, after a finite sequence of basis changes and restrictions, the real part of $A_0$ is positive definite and the other $A_j$ are skew-adjoint. The second main result is a solution to a noncommutative version of Hilbert's 17th problem: a positive regular noncommutative rational function is a sum of squares.
\fi

\maketitle

%%%%%%%%%%%%%%%%%%%%%%%%%%%%%%%%%%%%%%%%%%%%%%%%%%%%%%%%%%%%%%%%%%%%%%%%%%%%
%%%%%%%%%%%%%%%%%%%%%%%%%%%%%%%%%%%%%%%%%%%%%%%%%%%%%%%%%%%%%%%%%%%%%%%%%%%%
%%%%%%%%%%%%%%%%%%%%%%%%%%%%%%%%%%%%%%%%%%%%%%%%%%%%%%%%%%%%%%%%%%%%%%%%%%%%

\section{Introduction}

Let $\kk$ be the field of real or complex numbers and $\ulx=(x_1,\dots,x_g)$ a tuple of freely noncommuting variables. By the theory of division rings \cite{Ami,Coh,Reu}, the free algebra $\px$ of noncommutative polynomials admits a universal skew field of fractions $\rx$, whose elements are called noncommutative rational functions. They are usually represented with syntactically valid expressions involving $x_1,\ldots,x_g, +, \cdot, (, ), {}^{-1}$ and elements from $\kk$. Noncommutative rational functions play a prominent role in a wide range of areas. In ring theory, they appear as quasideterminants of matrices over noncommutative rings \cite{GKLLRT} and in the context of rings satisfying rational identities \cite{Ber}. In theoretical computer science, recognizable series of weighted automata are precisely formal power series expansions of noncommutative rational functions \cite{BR}. For similar reasons they emerge as transfer functions of linear systems evolving along free semigroups in control theory \cite{BGM}. These linear systems techniques are also applied in free probability for computing asymptotic eigenvalue distributions of noncommutative rational function evaluations on random matrices \cite{BMS}. In free analysis they are noncommutative analogs of meromorphic functions and are endowed with the difference-differential calculus \cite{KVV2,AM}. Finally, ensembles of noncommutative rational functions are natural maps between noncommutative semialgebraic domains in free real algebraic geometry \cite{HMV,BPT}.

While the interest in noncommutative rational functions originated from the universal property of the free skew field $\rx$, their importance in aforementioned areas derives from properties of their matrix evaluations. Here one evaluates a given noncommutative rational function $\rr$ on tuples of self-adjoint matrices, which leads to the notion of the domain of $\rr$. Minimal factorizations \cite{KVV1} and the extent of matrix convexity \cite{HMV} of a noncommutative rational function $\rr$ are examples of problems that directly depend on knowing the domain of $\rr$. On a more applied side, understanding of the singularities of (non)commutative rational functions is also important in control theory, e.g. for stability questions \cite{GKVVW} or controllability and observability of linear systems \cite{WSCP}. In this paper we analyze two properties of noncommutative rational functions arising from their evaluations, namely regularity and positivity.

A noncommutative rational function $\rr$ is {\it regular} if its domain contains all tuples of like sized self-adjoint matrices. For example, $(1 + x_1^2)^{-1}$ and $(1+x_1^2x_2^2)^{-1}$ are regular functions\footnote{
If $S_1$ and $S_2$ are positive semidefinite matrices, then the eigenvalues of $S_1S_2$ are real and nonnegative, so $I+S_1S_2$ is invertible. 	
}, as well as are all noncommutative polynomials. In general, checking the regularity of $\rr$ is harder if the representatives of $\rr$ are complicated, e.g. if they contain numerous nested inverses. Furthermore, we note that, as in the commutative case, further difficulties arise because singularities of a given rational expression might be removable. The proper tool for investigating regularity comes from automata and control theory: every noncommutative rational function admits a {\it linear systems realization}
\begin{equation}\label{e:00}
\rr=\cc^* L^{-1}\bb,
\end{equation}
where $\bb,\cc\in\kk^d$ and $L$  is a {\it linear matrix pencil} of size $d$: $L=A_0+\sum_jA_jx_j$ with $A_j\in\mat{d}$. For the existence we refer to \cite[Sections 4.2 and 6.2]{Coh} or \cite{BGM} for the case $A_0=I$. Linear pencils give rise to linear matrix inequalities $L(x)\succeq0$ and are thus ubiquitous in optimization \cite{WSV}, systems engineering \cite{SIG} and in real algebraic geometry, see e.g. determinantal representations of polynomials \cite{Bra,NT}, the solution of the Lax conjecture \cite{HV}, and the solution of the Kadison-Singer paving conjecture \cite{MSS}. If the linear pencil $L$ is of minimal size satisfying \eqref{e:00}, then the ``no hidden singularities theorem'' \cite[Theorem 3.1]{KVV1} implies that $\rr$ is regular if and only if every evaluation of $L$ on a tuple of self-adjoint matrices is nonsingular. Characterization of regular functions thus turns into a problem of recognizing everywhere invertible pencils.

After describing regular functions we address their positivity. We say that a noncommutative rational function $\rr$ is {\it positive} if $\rr(X)$ is positive semidefinite for every tuple of self-adjoint matrices $X$ in the domain of $\rr$. For example, $x_1^{-2}$ and $x_2^2-x_2x_1(1+x_1^2)^{-1}x_1x_2$ are positive functions. We solve the analog of Hilbert's 17th problem for regular noncommutative rational functions.
The original solution by Artin, stating that a nonnegative commutative polynomial is a sum of squares of rational functions (see e.g. \cite{BCR,Mar,DAn}), has been extended to the noncommutative setting in various ways \cite{PS,Hel,McC}. For example, Helton \cite[Theorem 1.1]{Hel} showed that every positive noncommutative polynomial is a sum of hermitian squares $\sum_kq_k^*q_k$, where $q_k$ are noncommutative polynomials. More general results about noncommutative polynomials that are positive semidefinite on certain free semialgebraic sets are now commonly known as noncommutative Positivstellens\"atze \cite{HMP,Nel,HKM}; for positivity results on free analytic functions see e.g. \cite{PTD,BKV}.

\subsection{Main results and reader's guide}

In Section \ref{sec2} we characterize everywhere invertible pencils. For example, if $A_j$ are skew-adjoint matrices for $1\le j\le g$ and $A_0$ is a sum of a positive definite and a skew-adjoint matrix, then $\Lambda(X)=A_0\otimes I+\sum_{j>0} A_j\otimes X_j$ is clearly nonsingular for every tuple $X$ of self-adjoint matrices. The condition on the coefficients of $\Lambda$ can be also stated as
$$\real(\Lambda)=\real(\Lambda(0))\succ0,$$
where $\real Y=\frac12(Y+Y^*)$ denotes the real part of a matrix. More generally, we say that a linear pencil $L$ is {\it \privileged} if there exist constant matrices $D_1,\dots,D_\ell,V_1,\dots,V_{\ell-1}$ and $V_\ell=0$ of appropriate sizes such that
$$\real(D_kLV_1\cdots V_{k-1})=\real(D_kL(0)V_1\cdots V_{k-1})\succeq0$$
for $1\le k\le\ell$ and the columns of $V_k$ form a basis of $\ker \real(D_kL(0)V_1\cdots V_{k-1})$. See also Definition \ref{d:Dcond} for a recursive version. The pencil $\Lambda$ described previously is \privileged with $\ell=1$ and $D_1=I$. The name refers to elliptic systems of partial differential equations \cite{Mir,GB} and is justified in our main result on linear pencils.

\begin{thmA}\label{t:a}
A pencil $L$ is \privileged if and only if $L(X)$ is of full rank for every self-adjoint tuple $X$.
\end{thmA}

More precise statements involving size bounds are given in Proposition \ref{p:privpen1} and Theorem \ref{t:privpen2}. For square pencils $L$, $L(X)$ is always invertible if and only if the free locus of $L$, defined in \cite{KV}, does not contain any self-adjoint tuples. Theorem \ref{t:a} can be therefore seen as a 
weak real Nullstellensatz for linear pencils. In Section \ref{sec3} we apply Theorem \ref{t:a} to regular noncommutative rational functions via the realization theory. Among regular functions we also describe {\it \stably bounded functions} $\rr$, i.e., those for which there exist $\ve,M>0$ such that for every (not necessarily self-adjoint) tuple $X$ satisfying $\|X^*-X\|<\ve$ we have $\|\rr(X)\|\le M$.

\begin{thmA}\label{t:b}
Let $\rr\in\rx$. Then $\rr$ is regular if and only if $\rr=\cc^*L^{-1}\bb$ for some \privileged pencil $L$. Furthermore, $\rr$ is \stably bounded if and only if $\rr=\cc^*(A_0+\sum_jA_jx_j)^{-1}\bb$, where $\real A_0$ is positive definite and $A_j$ are skew-adjoint for $j>0$.
\end{thmA}

See Theorem \ref{t:privreal} for the proof. In Section \ref{sec4} we address positivity of noncommutative rational functions.  We prove the following analog of Helton's sum of squares theorem \cite{Hel} for regular noncommutative rational functions.

\begin{thmA}\label{t:c}
Let $\rr\in\rx$ be regular. Then $\rr(X)$ is positive semidefinite for every tuple $X$ of self-adjoint matrices if and only if $\rr$ is a sum of hermitian squares of regular functions in $\rx$.
\end{thmA}

This statement is proved as Theorem \ref{t:nonneg} using a Hahn-Banach separation argument for a convex cone in a finite-dimensional vector space constructed from a noncommutative rational function. Lastly, we discuss the algorithmic perspective and present examples in Section \ref{sec5}.

\subsection*{Acknowledgments}

The authors thank Bill Helton for drawing the connection to elliptic systems of PDEs to their attention.

%%%%%%%%%%%%%%%%%%%%%%%%%%%%%%%%%%%%%%%%%%%%%
%%%%%%%%%%%%%%%%%%%%%%%%%%%%%%%%%%%%%%%%%%%%%

\section{Full rank pencils}\label{sec2}

In this section we prove that a linear pencil is \privileged if and only if it is of full rank on all self-adjoint tuples (Theorem \ref{t:privpen2}).

\subsection{Basic notation}
Throughout the paper let $\kk\in\{\R,\C\}$ and fix $g\in\N$. Let $\px=\kk\Langle x_1,\dots,x_g \Rangle$ be the free $\kk$-algebra of noncommutative polynomials in freely noncommuting variables $x_1,\dots,x_g$. We endow $\px$ with the involution $*$ satisfying $x_j^*=x_j$ and $\alpha^*=\bar{\alpha}$ for $\alpha\in\kk$.

Let $*$ be the involution on $\mat{n}$ given by the transpose (if $\kk=\R$) or the conjugate transpose (if $\kk=\C$) and let $\herm{n}\subseteq\mat{n}$ be the $\R$-subspace of self-adjoint matrices. The notation $A\succ0$ or $A\succeq0$ for $A\in \herm{n}$ means that $A$ is positive definite or positive semidefinite, respectively, while $\|\cdot\|$ always refers to the operator norm. Furthermore denote
$$\all=\bigcup_n \mat{n}^g,\qquad \allh=\bigcup_n \herm{n}^g$$
and
$$\real X=\frac{1}{2}(X+X^*),\qquad \imag X=
\left\{
\begin{array}{lr}
\frac12(X-X^*) & \ \text{if }\ \kk=\R \\[1mm]
\frac{1}{2i}(X-X^*) & \ \text{if }\ \kk=\C 
\end{array}\right.
$$
for $X\in\all$.

\subsubsection{Linear matrix pencils}

If $A_0,\dots,A_g\in \mat{d\times e}$, then
\begin{equation}\label{eq:pencil}
L=A_0+\sum_{j=1}^g A_jx_j\in \mat{d\times e}\otimes \px
\end{equation}
is a (rectangular) {\bf pencil} of size $d\times e$. It can be naturally evaluated on $\all$ as
$$L(X)=A_0\otimes I+\sum_{j=1}^g A_j\otimes X_j\in \mat{dn\times en}$$
for $X\in\mat{n}^g$.

\subsection{\Privileged pencils}

\begin{defn}\label{d:Dcond}
Let $d\ge e$ and $L=A_0+\sum_jA_jx_j$ with $A_j\in \mat{d\times e}$.
\begin{enumerate}[label={\rm(\arabic*)}]
	\item $L$ is {\bf \stably \privileged} if there exists $D\in \mat{e\times d}$ such that
	$$\real (DA_0)\succ0,\qquad \real(DA_j)=0 \quad\text{for } j>0.$$
	\item With respect to $e$ we recursively define $L$ to be {\bf \privileged} if
	\begin{enumerate}
		\item it is \stably \privileged; or
		\item there exists $D\in \mat{e\times d}$ such that
		$$0\neq\real(DA_0)\succeq0,\qquad \real(DA_j)=0 \quad\text{for } j>0$$
		and $LV$ is \privileged, where columns of $V$ form a basis for $\ker \real(DA_0)$. Note that $LV$ is a pencil of size $d\times e'$ with $e'<e$.
	\end{enumerate}
\end{enumerate}
A pencil of size $d\times e$ with $d<e$ is (\stably) \privileged if and only if $L^*$ is (\stably) \privileged.
\end{defn}

\begin{exa}
Let $\kk=\R$, $g=2$ and
$$L=\begin{pmatrix}
1 & x_1-x_2 & x_1-1 \\
x_2-x_1 & 1 & 1 \\
1-x_1 & -1 & 0
\end{pmatrix}=A_0+A_1x_1+A_2x_2.$$
It is easy to check that every $3\times3$ matrix $D$ satisfying $\real(DA_1)=\real(DA_2)=0$ is a scalar multiple of $I$, so $L$ is not \stably \privileged. However, we have
$$\real(A_0)=\diag(1,1,0),\qquad \real(A_1)=\real(A_2)=0.$$
Restricting to the kernel of $\real(A_0)$ we obtain
$$V=\begin{pmatrix}
0 \\ 0 \\ 1
\end{pmatrix},\qquad
L'=LV=\begin{pmatrix}
x_1-1 \\ 1 \\ 0
\end{pmatrix}.$$
Choosing $D'=(0 \ 1 \ 0)$ we get $D'L'=1$, hence $L'$ is \stably \privileged and $L$ is \privileged.
\end{exa}

\begin{rem}
The terminology of Definition \ref{d:Dcond} refers to the ellipticity of partial differential equations  \cite{Nir,Mir,GB}. For example, a first order system
$$\sum_{j=1}^g P_j(x) \frac{\partial u}{\partial x_j}=f(x,u),$$
where $P_j$ are $d\times e$ matrices with $d\ge e$, is elliptic at the point $x$ if the matrix
$$P(x,\xi)=\sum_j^g P_j(x)\xi_j$$
has rank $e$ for all $\xi\in\R^g\setminus\{0\}$; see \cite[Section 4.7]{GB}. The analogy between free elliptic pencils and elliptic systems becomes clear in Theorem \ref{t:privpen2} where we prove that a pencil $L$ is \privileged if and only if $L(X)$ is of full rank for all $X\in\allh$.
\end{rem}

\begin{prop}\label{p:privpen1}
A pencil $L$ 
of size $d\times e$ with $d\geq e$
is \stably \privileged if and only if for some $\ve>0$,
\begin{equation}\label{e:20}
L(X)^*L(X)\succ\ve I \qquad \text{for all }\  X\in\allh.
\end{equation}
Furthermore, it suffices to test \eqref{e:20} for $X$ of size at most $(g+1)e^2$. 
\end{prop}

\begin{proof}
Let $L=A_0+\sum_jA_jx_j$ with $A_j\in \mat{d\times e}$ and $d\ge e$.

$(\Rightarrow)$ Let $D$ be a matrix with $\real (DA_0)=R^*R$ for $R\in\GL_d(\kk)$ and $\real(DA_j)=0$ for $j>0$. Denote $K=R^{-*}\imag (DL)R^{-1}$ and let $\kappa=1$ if $\kk=\R$ and $\kappa=i$ if $\kk=\C$. If $X\in \herm{n}^g$ and $\vv\in\kk^{en}$, then
\begin{align*}
\|D\|^2\langle L(X)\vv,L(X)\vv\rangle
&\ge \langle (DL)(X)\vv,(DL)(X)\vv\rangle\\
&=\langle (R^*R+\kappa R^*KR)(X)\vv,(R^*R+\kappa R^*KR)(X)\vv\rangle\\
&\ge \|R^{-*}\|^{-2}\langle (I+\kappa K)(X)(R\otimes I)\vv,(I+\kappa K)(X)(R\otimes I)\vv\rangle\\
&\ge \|R^{-*}\|^{-2}\langle (R\otimes I)\vv,(R\otimes I)\vv\rangle\\
&\ge \|R^{-*}\|^{-4} \langle \vv,\vv \rangle
\end{align*}
since $\kappa K$ is skew-adjoint on $\allh$. Hence we can take $\ve=\|R^{-1}\|^{-4}\|D\|^{-2}$.

$(\Leftarrow)$ Since $A_0$ is of full rank, there exists $R\in \GL_d(\kk)$ such that $RA_0$ is an isometry, i.e., $(RA_0)^*(RA_0)=I$. Since
$$\langle (RL)(X)\vv,(RL)(X)\vv\rangle\ge \|R^{-1}\|^{-2}\langle L(X)\vv,L(X)\vv\rangle$$
for every $X\in \herm{n}^g$ and $\vv\in\kk^{en}$, we have $(RL)(X)^*(RL)(X)\succ\ve\|R^{-1}\|^{-2} I$.

Without loss of generality we can thus assume that $A_0$ is an isometry. Also let
$$\cK=\left(\sum_{j>0}\im A_j\right)^{\perp}.$$
Because $L^*L-\ve I$ is a positive polynomial (on matrices of size at most $(g+1)e^2$), it is a sum of hermitian squares of matrix-valued polynomials of degree at most 1 by \cite[Theorem 0.2]{McC} or \cite[Theorem 1.1]{MP}:
$$L^*L-\ve I=\sum_{k=1}^N\left(C_{k,0}+\sum_j C_{k,j}x_j\right)^*\left(C_{k,0}+\sum_j C_{k,j}x_j\right),
\qquad C_{k,j}\in\mat{e}.$$
If $\bC=(C_{k,j})_{j,k}\in \mat{e}^{N\times(g+1)}$ and $\ell^*=(1,x_1,\dots,x_g)$, then
$$L^*L-\ve I=\ell^* \bC^*\bC \ell.$$
By looking at the coefficients of $L$ we conclude that the positive semidefinite matrix $\bA=\bC^*\bC$ is of the form
\begin{equation}\label{e:01}
\bA=\begin{pmatrix}
I-\ve I& B_1 & B_2 & \cdots & B_g \\
B_1^* & A_1^*A_1 & A_1^*A_2 &\cdots & A_1^*A_g\\
B_2^* & A_2^*A_1 & A_2^*A_2 &\cdots & A_2^*A_g\\
\vdots & \vdots & \vdots &\ddots & \vdots\\
B_g^* & A_g^*A_1& A_g^*A_2 &\cdots & A_g^*A_g
\end{pmatrix}
\end{equation}
where $B_j\in \mat{e}$ satisfy
\begin{equation}\label{e:02}B_j+B_j^*=A_0^*A_j+A_j^*A_0.\end{equation}
Let $\vv\in \kk^{gd}$ be arbitrary and set $\ww=0\oplus \vv\in\kk^{(g+1)d}$. If $\vv$ satisfies
$$\begin{pmatrix}A_1 & \cdots & A_g\end{pmatrix}\vv=0,$$
then $\ww^*\bA\ww=0$ and therefore $\bA\ww=0$, so
$$\begin{pmatrix}B_1 & \cdots & B_g\end{pmatrix}\vv=0.$$
Hence the rows of a block matrix $(B_1 \ \cdots \ B_g)$ lie in the linear span of the rows of $(A_1 \ \cdots \ A_g)$, so there exists $T\in \mat{e\times d}$ such that
\begin{equation}\label{e:03}
\begin{pmatrix}B_1 & \cdots & B_g\end{pmatrix}=T \begin{pmatrix}A_1 & \cdots & A_g\end{pmatrix},\qquad T\cK=0.
\end{equation}
Since $L(0)^*L(0)-\ve I\succ0$ and $\bA\succeq0$, the Schur complement of $I-\ve I$ in $\bA$
\begin{equation}\label{e:04}
\begin{pmatrix}A_1^* \\ \vdots \\ A_g^*\end{pmatrix} \begin{pmatrix}A_1 & \cdots & A_g\end{pmatrix}-
(1-\ve)^{-1}\begin{pmatrix}B_1^* \\ \vdots \\ B_g^*\end{pmatrix} \begin{pmatrix}B_1 & \cdots & B_g\end{pmatrix}
\end{equation}
is also positive semidefinite. Combining \eqref{e:03} and \eqref{e:04} yields
\begin{equation}\label{e:05}I-(1-\ve)^{-1}T^*T\succeq0.\end{equation}
Let $D=A_0^*-T$. Now \eqref{e:02} and \eqref{e:03} imply $\real(DA_j)=0$ for $j>0$, while \eqref{e:05} together with $A_0^*A_0=I$ yields
$$0\preceq A_0^*\left((1-\ve)I-T^*T\right)A_0=2\real(DA_0)-(\ve I+A_0^*D^*DA_0),$$
and therefore $\real(DA_0)\succ0$.
\end{proof}

\begin{lem}\label{l:small}
Let $\uu_1,\dots,\uu_n\in\kk^d$ and $\vv_1,\dots,\vv_n\in\kk^m$. If $(\ww^*\otimes I)(\sum_k\uu_k\otimes \vv_k)=0$ for all $\ww\in\kk^d$, then $\sum_k\uu_k\otimes \vv_k=0$.
\end{lem}

\begin{proof}
Without loss of generality we can assume that $\vv_k$ are linearly independent. If
$$\sum_k(\ww^*\uu_k)\vv_k=0,$$
then $\ww^*\uu_k=0$ for all $k$. Since this holds for every $\ww$, we have $\uu_k=0$ and hence $\sum_k\uu_k\otimes \vv_k=0$.
\end{proof}

The proof of the following theorem applies a specialized GNS construction that is inspired by a more general and intricate version in the proof of the matricial real Nullstellensatz in \cite{Nel}.

\begin{thm}\label{t:privpen2}
For a pencil $L=A_0+\sum_jA_jx_j$ with $A_j\in \mat{d\times e}$ and $d\ge e$, the following are equivalent:
\begin{enumerate}[label={\rm(\roman*)}]
	\item $L$ is \privileged;
	\item $L$ is of full rank on $\allh$;
	\item $L(X)$ is of full rank for all $X\in\allh$ of size at most $(g+1)e^2$.
\end{enumerate}
\end{thm}

\begin{rem}
Square linear pencils that are nonsingular on whole $\all$ been characterized in \cite[Corollary 3.4]{KV}: if $L=A_0+\sum_jA_jx_j$, then $\det L(X)\neq0$ for all $X\in\all$ if and only if $A_0^{-1}A_1,\dots,A_0^{-1}A_g$ are jointly nilpotent matrices.
\end{rem}

\begin{rem}
In the opposite direction, square linear pencils $L$ with $\det L(X)=0$ for all $X\in\allh$ are precisely those that are not invertible as matrices over $\rx$. By \cite[Corollary 6.3.6]{Coh}, a linear pencil $A_0+\sum_jA_jx_j$ with $A_j\in\mat{d}$ is not invertible over $\rx$ if and only if there exist matrices $U,V\in\GL_d(\kk)$ such that for $0\le j\le g$ we have a block decomposition
$$UA_jV=\begin{pmatrix} * & 0 \\ * & * \end{pmatrix},$$
where the zero block is of size $d_1\times d_2$ with $d_1+d_2>d$. A linear bound on size of $X$ for testing $\det L(X)=0$ has been given in \cite{DM}.
\end{rem}

\begin{proof}[Proof of Theorem \ref{t:privpen2}]
We prove (i)$\Rightarrow$(ii) and (iii)$\Rightarrow$(i) by induction on $e$, while (ii)$\Rightarrow$(iii) is obvious.

(i)$\Rightarrow$(ii) By Proposition \ref{p:privpen1}, the claim holds if $L$ is \stably \privileged. If $L$ is \privileged but not \stably \privileged there exists $D\in \mat{e\times d}$ such that $0\neq\real (DA_0)\succeq0$, $\real(DA_j)=0$ for $j>0$ and $LV$ is \privileged, where columns of $V$ constitute a basis of $\ker \real (DA_0)$. Let $X\in \herm{n}^g$ and consider the decomposition
$$\ker L(X)\subseteq \kk^{en}=\kk^e\otimes\kk^n=((\ker \real (DA_0))^\perp\otimes\kk^n) \oplus (\ker \real (DA_0)\otimes\kk^n).$$
If $\begin{psmallmatrix}\uu \\ \vv\end{psmallmatrix}\in\ker L(X)$, then
$$\begin{pmatrix}\uu^* & \vv^*\end{pmatrix}\real(DL)(X)\begin{pmatrix}\uu \\ \vv\end{pmatrix}=0$$
and so $\uu^*(\real (DA_0)\otimes I)\uu^*=0$. Hence $\uu=0$ and $(LV)(X)\vv=0$, therefore the induction hypothesis implies $\vv=0$. Thus $L$ is of full rank on $\allh$.

(iii)$\Rightarrow$(i) Assume that $L$ is not \privileged. Therefore for every $D$ such that $0\neq\real (DA_0)\succeq0$ and $\real(DA_j)=0$ for $j>0$, $LV$ is not \privileged, where $V$ consists of a basis of $\ker\real(DA_0)$. By assumption there exists $X\in\allh$ such that $(LV)(X)$ is not of full rank, so $L(X)$ is not of full rank. Hence we need to consider the situation when $L$ satisfies
\begin{equation}\label{e:07}
\forall D\colon \quad \real (DA_0)\succeq0 \text{ and } \real(DA_j)=0 \text{ for } j>0 \ \Rightarrow\ \real(DA_0)=0. 
\end{equation}
For $k\in\{0,1,2\}$ define
$$\cV_k=\left\{p\in \mat{e}\otimes\px\colon \deg\le k\right\},\quad \cU=\mat{e\times d}L+L^*\mat{d\times e},$$
$$\cVsa_2=\left\{p\in\cV_2\colon p^*=p\right\},\quad \cS_2=\left\{\sum_jp_j^*p_j\colon p_j\in \cV_1\right\}.$$
Here $\cV_k$ and $\cU$ are $\kk$-linear spaces, while $\cVsa_2$ and $\cS_2$ are a $\R$-linear space and a convex cone, respectively. It is easy to verify that $\cS_2$ is closed in $\cV_2$ (see e.g.~the proofs of \cite[Proposition 3.4]{MP} or Proposition \ref{p:closed} below). Observe that $\eqref{e:07}$ implies $\cU\cap \cS_2=\{0\}$. Indeed, if $DL+L^*E^*\in \cS_2$, then $(DL+L^*E^*)+(EL+L^*D^*)\in \cS_2$ and so $\real ((D+E)A_0)\succeq0$ and $\real((D+E)A_j)=0$ for $j>0$. Hence \eqref{e:07} implies $\real ((D+E)A_0)=0$ and consequently $DL+L^*E^*=0$.

By \cite[Theorem 2.5]{Kle} there exists a $\R$-linear functional $\lambda:\cVsa_2+\cU\to \R$ such that $\lambda(\cS_2\setminus\{0\})\subseteq \R_{>0}$ and $\lambda(\cU)=\{0\}$, which we extend to a $\kk$-linear functional $\Lambda:\cV_2\to\kk$ by setting $\Lambda(p)\coloneq\lambda(\real p)+i\lambda(\imag p)$ if $\kk=\C$ and $\Lambda(p)\coloneq \lambda(\real p)$ if $\kk=\R$. Consequently we obtain a scalar product $\langle p_1,p_2\rangle\coloneq\Lambda(p_2^*p_1)$ on $\cV_1$. Note that
$$\langle a p,q\rangle=\Lambda(q^*ap)=\langle p,a^*q\rangle$$
for all $a\in\cV_0$ and $p,q\in\cV_1$.
Let $\pi:\cV_1\to\cV_0$ be the orthogonal projection. For every $a,b\in\cV_0$ and $1\le j\le g$ we have
$$\langle \pi(a x_j),b\rangle=\langle a x_j,b\rangle=\langle x_j,a^*b\rangle=\langle\pi(x_j),a^*b\rangle=\langle a\pi(x_j),b\rangle,$$
so
\begin{equation}\label{e:09}
\pi(ap)=a\pi(p) \qquad \forall a\in\cV_0,\ p\in\cV_1.
\end{equation}
Now define
\begin{alignat*}{3}
X_j&\colon\cV_0\to\cV_0, &\qquad & b \mapsto \pi(x_j b), \\
\ell_a&\colon\cV_0\to\cV_0, &\qquad & b \mapsto a b.
\end{alignat*}
It is easy to check that $X_j$ is a self-adjoint operator that commutes with $\ell_a$ by \eqref{e:09}. Let $D\in\mat{e\times d}$ be arbitrary and consider $I\in\mat{e}$ as a vector in $\cV_0$. Then $DL$ determines a linear operator $(DL)(X):\cV_0\to\cV_0$ and
$$\langle (DL)(X)I,b\rangle=\langle \pi(DL),b\rangle=\langle DL,b\rangle=\Lambda((b^*D)L)=0$$
for every $b\in \cV_0$ by \eqref{e:09} and the definition of $\Lambda$. Therefore $(DL)(X)I=0$ and consequently $L(X)I=0$ by Lemma \ref{l:small}.

Finally, the bound from the statement follows from Proposition \ref{p:privpen1} and the fact that $\dim\cV_0=e^2<(g+1)e^2$.
\end{proof}

\begin{rem}
Let $L=A_0+\sum_{j>0}A_jx_j$ be given and assume $D$ satisfies
\begin{equation}\label{e:sdp}
\real (DA_0)\succeq0,\qquad \real(DA_j)=0 \quad\text{for } j>0.
\end{equation}
If $\real (DA_0)\neq0$, then Theorem \ref{t:privpen2} implies that $L$ is \privileged if and only if $LV$ is \privileged, where $V$ comprises a basis of $\ker\real(DA_0)$. This fact simplifies the ellipticity testing: we can do the recursion with an arbitrary $D$ which non-trivially solves \eqref{e:sdp} in the sense that $\real (DA_0)\neq0$.
\end{rem}

%%%%%%%%%%%%%%%%%%%%%%%%%%%%%%%%%%%%%%%%%%%%%
%%%%%%%%%%%%%%%%%%%%%%%%%%%%%%%%%%%%%%%%%%%%%

\section{Regular rational functions}\label{sec3}

In this section we turn our attention to regular nc rational functions, i.e., those without singularities. 
The main result, Theorem \ref{t:privreal}, shows that
$\rr\in\rx$ is regular (\stably bounded) if and only if it admits a realization with a (\stably) \privileged pencil.

\subsection{Preliminaries}

We introduce noncommutative rational functions using matrix evaluations of formal rational expressions following \cite{HMV,KVV2}. Originally they were defined ring-theoretically, cf.~\cite{Ami,Coh}. {\bf Noncommutative (nc) rational expressions} are syntactically valid combinations of elements in $\kk$, freely noncommuting variables $\{x_1,\dots,x_g\}$, arithmetic operations $+,\cdot,{}^{-1}$ and parentheses $(,)$. For example, $(1+x_2^{-1}x_1)^{-1}+1$, $x_1+(-1)x_1$ and $0^{-1}$ are nc rational expressions. Their set is $\re$.

Given $r\in\re$ and $X\in \mat{n}^g$, the evaluation $r(X)$ is defined in the obvious way if all inverses appearing in $r$ exist at $X$. The set of all $X\in\all$ such that $r$ is defined at $X$ is is called the {\bf domain of $r$} and denoted $\dom r$. Note that $\dom r\cap\mat{n}^g\subseteq \mat{n}^g$ is a Zariski open set for every $n\in\N$ and therefore either empty or dense in $\mat{n}^g$ with respect to Euclidean topology. A nc rational expression $r$ is {\bf non-degenerate} if $\dom r\neq\emptyset$. On the set of all non-degenerate nc rational expressions we define an equivalence relation $r_1\sim r_2$ if and only if $r_1(X)=r_2(X)$ for all $X\in\dom r_1\cap \dom r_2$. The equivalence classes with respect to this relation are called {\bf noncommutative (nc) rational functions}. By \cite[Proposition 2.1]{KVV2} they form a skew field denoted $\rx$, which is the universal skew field of fractions of $\px$ by \cite[Section 4.5]{Coh}. The equivalence class of a nc rational expression $r\in\re$ is written as $\rr\in\rx$. The previously defined involution on $\px$ naturally extends to $\rx$.

We define the {\bf domain} of a nc rational function $\rr\in\rx$ as the union of $\dom r$ over all representatives $r\in\re$ of $\rr$. Lastly, let
$$\domh r=\dom r\cap\allh,\qquad \domh \rr=\dom \rr\cap\allh.$$
For a non-degenerate $r\in\re$ we have $\domh r\neq\emptyset$ if and only if $\dom r\neq\emptyset$, see e.g.~\cite[Remark 6.8]{Vol}.

\subsubsection{Realizations of nc rational functions}

For every $\rr\in\rx$ there exist $d\in\N$, $\bb,\cc\in\kk^d$ and a linear pencil $L$ of size $d$ such that
\begin{equation}\label{e:11}
\rr=\cc^* L^{-1}\bb,
\end{equation} cf.~\cite[Section 4.2]{Coh}. We say that \eqref{e:11} is a {\bf realization} of $\rr$ of size $d$; we refer to \cite{Coh,BR} for good expositions on classical realization theory.

Fix $\rr\in\rx$ and suppose $0\in\dom\rr$. In general, $\rr$ admits various realizations. A realization of $\rr$ whose size is smallest among all realizations of $\rr$ is called {\bf minimal}. These are unique up to basis change 
\cite[Theorem 2.4]{BR}, and if $\cc^*L^{-1}\bb$ is a minimal realization of $\rr$, 
then $\dom\rr\subseteq\left\{X\in \all\colon \det L(X)\neq0\right\}$, see \cite[Theorem 3.1]{KVV1}.

\subsection{Regularity}

\begin{defn}\label{d:funherm}
We say that $\rr\in\rx$ is:
\begin{enumerate}[label={\rm(\arabic*)}]
\item {\bf regular} if $\domh \rr= \allh$;
\item {\bf bounded} if there exists $M>0$ such that $\|\rr(X)\|\le M$ for all $X\in\domh \rr$.
\item {\bf \stably bounded} if there exist $\ve>0$ and $M>0$ such that $\| \rr(X)\|\le M$ for all $X\in\dom \rr$ with $\|\imag X\|<\ve$.
\end{enumerate}
Analogously, we say that $r\in\re$ is {\bf regular} if $\domh r= \allh$.
\end{defn}

This definition is naturally extended to matrices over $\rx$. Obviously a regular expression yields a regular function and (3) implies (2). Using Riemann’s removable singularities theorem \cite[Theorem 7.3.3]{Kra} it is not hard to deduce that (3) implies (1); however, this is also a consequence of Theorem \ref{t:privreal}.

\begin{exa}
Examples of regular but not bounded nc rational functions are nonconstant nc polynomials. An example of a bounded but not regular (and hence also not \stably bounded) function is $\rr=(1+x_1x_2^{-2}x_1)^{-1}$: indeed, we have $\|\rr(X_1,X_2)\|\le1$ for all $(X_1,X_2)\in\domh\rr$ and $(0,0)\notin \dom\rr$. On the other hand, $(\frac12+x_1^2+x_2^2+x_1^2x_2^2)^{-1}$ is an example of a \stably bounded rational function.
\end{exa}

\begin{prop}\label{p:expr}
Let $M$ be a square matrix over $\rx$ and assume each of its entries admits a regular rational expression. If $M(X)$ is nonsingular for every $X\in\allh$, then every entry of $M^{-1}$ admits a regular expression.
\end{prop}

\begin{proof}
We prove the statement by induction on the size $d$ of $M$. If $d=1$, then $M=\rr$ is an everywhere invertible regular rational function with a corresponding regular expression $r$; hence $M^{-1}$ is given by $r^{-1}$.
	
Assume the statement holds for matrices of size $d-1$ and let $\mm$ be the first column of $M$. Then $\mm(X)$ is of full rank for every $X\in\allh$, so the regular rational function $\mm^*\mm$ is everywhere invertible and its inverse is given by a regular expression. Hence the entries of the Schur complement of $\mm^*\mm$ in $M^*M$ admit regular rational expressions. Since $M^*M$ is nonsingular on $\allh$, the same holds for the Schur complement, which is a matrix of size $d-1$. By the induction hypothesis, the entries of the inverse of this Schur complement admit regular rational expressions, hence the same holds for the inverse of $M^*M$. Finally, the entries of $M^{-1}=(M^*M)^{-1}M^*$ admit regular rational expressions.
\end{proof}

\begin{cor}\label{c:expr}
If $\rr\in\rx$ is regular, then it arises from a regular rational expression.
\end{cor}

\begin{proof}
Let $\rr=\cc^* L^{-1}\bb$ be a minimal realization of $\rr$. Since $\rr$ is regular, $L$ is nonsingular on $\allh$ by \cite[Theorem 3.1]{KVV1}. Since $L$ is a matrix of polynomials, entries of $L^{-1}$ admit regular rational expressions by Proposition \ref{p:expr}. Hence $\rr$ admits a regular rational expression.
\end{proof}

\begin{lem}\label{l:stb1}
Let $L$ be a square linear pencil. The following are equivalent:
\begin{enumerate}[label={\rm(\roman*)}]
\item $L^{-1}$ is \stably bounded;
\item $L^{-1}$ is bounded;
\item There exists $\eta>0$ such that $L(X)^*L(X)\succ\eta^2 I$ for all $X\in\allh$.
\end{enumerate}
\end{lem}

\begin{proof}
(i)$\Rightarrow$(ii) Trivial.

(ii)$\Rightarrow$(iii). Assume that $\|L^{-1}(X)\|\le M$ for $X\in\domh L^{-1}$; this means that the largest eigenvalue of $(L^*L)^{-1}(X)$ is at most $M^2$, so the smallest eigenvalue of $(L^*L)(X)$ is at least $\frac{1}{M^2}$. Since $\domh L^{-1}\cap \herm{n}^g$ is dense in $\herm{n}^g$ for infinitely many $n\in\N$, we conclude that the smallest eigenvalue of $(L^*L)(X)$ is at least $\frac{1}{M^2}$ for every $X\in \allh$ and hence $L(X)^*L(X)\succ \frac{1}{M^2} I$ for all $X\in\allh$.

(iii)$\Rightarrow$(i) Let $L=A_0+\sum_jA_jx_j$ and assume $L(X)^*L(X)\succ\eta^2 I$ for all $X\in\allh$. Choose $\ve=\frac{\eta}{2}(\sum_{j>0}\|A_j\|)^{-1}$. If $X\in\dom L^{-1}$ and $\|\imag X\|<\ve$, then
$$\left\|L(\real X)^{-1}\sum_{j>0}A_j\otimes\imag X_j\right\|\le \frac{1}{\eta}\left(\sum_{j>0}\|A_j\|\right)\ve =\frac12,$$
so
$$L(X)=L(\real X)^{-1}\left(I+L(\real X)^{-1}\sum_{j>0}A_j\otimes\imag X_j \right)$$
is invertible and 
$\|L(X)^{-1}\|\le \frac{2}{\eta}$.
\end{proof}

\begin{lem}\label{l:stb2}
A nc rational function is \stably bounded if and only if the inverse of the pencil from its minimal realization is \stably bounded.
\end{lem}

\begin{proof}
The implication $(\Leftarrow)$ is clear, so consider $(\Rightarrow)$. For $1\le j\le g$ let
$$\Delta_j:\rx\to\rx\otimes\rx$$
be the difference-differential operator as in \cite[Section 4]{KVV2}; this is the noncommutative counterpart of both 
the partial finite difference and partial differential operator. If $X,X'\in\dom\rr\cap\mat{n}^g$ and $\Delta(\rr)(X,X')$ is interpreted as an element of $\End_{\kk}(\mat{n})\cong \mat{n}\otimes \mat{n}$, then
$$\begin{pmatrix}X& W\\ 0& X'\end{pmatrix}\in\dom\rr\cap\mat{2n}^g$$
for every $W\in M_n(\kk)^g$ and
$$\rr\begin{pmatrix}X& W\\ 0& X'\end{pmatrix}$$
is up to conjugation by a permutation matrix equal to
$$\begin{pmatrix}\rr(X)& \sum_j\Delta(\rr)(X,X')(W_j)\\ 0& \rr(X')\end{pmatrix}$$
by \cite[Theorem 4.8]{KVV2}. In particular, if $\rr$ is \stably bounded, then $(\Delta_j\rr)(x,0)$ and $(\Delta_j\rr)(0,x)$ are also \stably bounded nc rational functions. Indeed, suppose $\|\rr(X)\|\le M$ for all $X\in\dom \rr$ with $\|\imag X\|<\ve$. Then for every $X\in\dom\rr$ with $\|\imag X\|<\frac{\ve}{2}$ we have
$$\left\|\imag\begin{pmatrix}X & \ve I \\\ 0& 0\end{pmatrix}\right\|<\ve$$
and therefore
$$\left\|\begin{pmatrix}\rr(X) & \ve(\Delta_j\rr)(X,0) \\ 0& \rr(0)\end{pmatrix}\right\|\le M$$
by assumption. Consequently
$$\|(\Delta_j\rr)(X,0)\|\le \frac{M}{\ve}.$$
Since $\dom\rr$ is dense in $\dom (\Delta_j\rr)(x,0)$, we conclude that $(\Delta_j\rr)(x,0)$ is \stably bounded.

Now let $\rr\in\rx$ be a \stably bounded nc rational function with minimal realization $\rr=\cc^* L^{-1}\bb$ which can be chosen such that $L(0)=I$. By \cite[Example 4.7]{KVV2} we have
$$\Delta_j(\rr)(x,0)=\cc^*L^{-1}A_j\bb,\qquad \Delta_j(\rr)(0,x)=\cc^*A_jL^{-1}\bb.$$
Minimality of the realization implies
$$\spa_{\kk}\{A^w\bb\colon w\in\mx\}=\kk^d,\qquad \spa_{\kk}\{(A^*)^w\cc\colon w\in\mx\}=\kk^d$$
by \cite[Proposition 2.1]{BR}. Therefore we conclude that every entry of $L^{-1}$ is \stably bounded.
\end{proof}

\subsection{\Privileged realizations and regular rational functions}

\begin{thm}\label{t:privreal}
Let $\rr\in \rx$.
\begin{enumerate}[label={\rm(\arabic*)}]
\item $\rr$ is \stably bounded if and only if it admits a (minimal) realization with a \stably \privileged pencil.
\item $\rr$ is regular if and only if it admits a (minimal) realization with a \privileged pencil.
\end{enumerate}	
\end{thm}

\begin{proof}
(1) If $\rr$ is \stably bounded and $\cc^* L^{-1}\bb'$ is its minimal realization, then $L^{-1}$ is \stably bounded by Lemma \ref{l:stb2}. Hence $L$ is \stably \privileged by Lemma \ref{l:stb1} and Proposition \ref{p:privpen1}. Conversely, if $\rr$ admits a realization with a \stably \privileged pencil $L$, then $L^{-1}$ is bounded by Proposition \ref{p:privpen1} and hence \stably bounded by Lemma \ref{l:stb1}. Therefore $\rr$ is \stably bounded.

(2) Let $\cc^* L^{-1}\bb$ be a minimal realization of $\rr$. Then $L$ is invertible on $\allh$ by \cite[Theorem 3.1]{KVV1} and hence \privileged by Theorem \ref{t:privpen2}. The converse implication also follows by Theorem \ref{t:privpen2}.
\end{proof}

\begin{rem}
One may also ask an analogous question about functions $\rr\in\rx$ satisfying $\dom\rr=\all$. The answer to this question is much simpler \cite[Theorem 4.2]{KV}: a nc rational function is defined at every point in $\all$ if and only if it is a nc polynomial.
\end{rem}

\subsection{Functions in \texorpdfstring{$\ulx$}{x} and \texorpdfstring{$\ulx^*$}{x*}}\label{ss:xx}

We briefly discuss nc rational functions in $\ulx$ and $\ulx^*$, 
i.e., elements of the free skew field with involution $\rxxc$. They are naturally evaluated at $g$-tuples of matrices $X$ by replacing $x_j$ with $X_j$ and $x_j^*$ with $X_j^*$. We refer to \cite{KS} for analytic properties of these $*$-evaluations. On the other hand, if $\uly$ is a copy of $\ulx$, then we have skew field isomorphisms
\begin{align*}
\rxxc\to \rxyc,\qquad & x_j\mapsto x_j+iy_j,\ x_j^*\mapsto x_j-iy_j\\
\rxyc\to \rxxc,\qquad & x_j\mapsto\frac12 (x_j+x_j^*),\ y_j\mapsto\frac1{2i} (x_j-x_j^*).
\end{align*}
Thus we get a natural correspondence between $*$-evaluations of elements in $\rxxc$ and Hermitian evaluations of elements in $\rxyc$. Our main results on rational functions and pencils in self-adjoint variables can be easily adapted to this setup. We leave this as an exercise for the reader.

%%%%%%%%%%%%%%%%%%%%%%%%%%%%%%%%%%%%%%%%%%%%%
%%%%%%%%%%%%%%%%%%%%%%%%%%%%%%%%%%%%%%%%%%%%%

\section{Positive rational functions}\label{sec4}

In this section we solve a noncommutative analog of Hilbert's 17th problem: every positive regular nc rational function is a sum of hermitian squares, see Theorem \ref{t:nonneg}. For this we shall require a description of complexity of nc rational expressions. A {\bf sub-expression} of $r\in\re$ is any nc rational expression which appears during the construction of $r$. For example, if $r=((2+x_1)^{-1}x_2)x_1^{-1}$, then all its sub-expressions are
$$2,x_1,2+x_1,(2+x_1)^{-1},x_2,(2+x_1)^{-1}x_2,x_1^{-1},((2+x_1)^{-1}x_2)x_1^{-1}$$
We recursively define a complexity-measuring function $\tau:\re\to\N_0$ as follows:
\begin{enumerate}[label=(\alph*)]
	\item $\tau(\alpha)=0$ for $\alpha\in\kk$;
	\item $\tau(x_j)=1$ for $1\le j\le g$;
	\item $\tau(r_1+r_2)=\max\{\tau(r_1),\tau(r_2)\}$ for $r_1,r_2\in\re$;
	\item $\tau(r_1r_2)=\tau(r_1)+\tau(r_2)$ for $r_1,r_2\in\re$;
	\item $\tau(r^{-1})=2\tau(r)$ for $r\in\re$.
\end{enumerate}
Note that there is also a well-defined map $r\mapsto r^*$ on $\re$ that mimics the involution on $\rx$ and $\tau(r^*)=\tau(r)$ for all $r\in\re$.

\subsection{A sum of squares cone associated with a rational expression}

Throughout the rest of this section fix a non-degenerate expression $r\in\re$ and the following notation. Let $Q\subset\re$ be the finite set of all sub-expressions of $r$ and their images under the map $q\mapsto q^*$. Then define $\tilde{Q}=\{\qq\colon q\in Q\}\subset\rx$ and
$$\cV_k=\sum_{j=0}^k \kk \overbrace{\tilde{Q}\cdots \tilde{Q}}^j,\qquad
\cVsa_k=\{\rs\in\cV_k\colon \rs=\rs^*\},\qquad
\cS_{2k}=\left\{\sum_j\rs_j^*\rs_j\colon \rs_j\in\cV_k\right\}$$
for $k\in\N$. Then $\cS_{2k}\subseteq \cVsa_{2k}\subseteq \cV_{2k}$ are a convex cone, $\R$-linear space and $\kk$-linear space, respectively. Since $\cVsa_{2k}$ is finite-dimensional, every norm on $\cVsa_{2k}$ yields the usual Euclidean topology. We note that $Q$ and $\tilde{Q}$ are rational analogs of Newton chips \cite{BKP} for free polynomials.

\begin{prop}\label{p:closed}
$\cS_{2k}$ is closed in $\cVsa_{2k}$.
\end{prop}

\begin{proof}
Since $\cVsa_{2k}$ is finite-dimensional, there exists $X\in\domh r$ such that
$$\forall \rs\in\cVsa_{2k}\colon\ \rs(X)=0\ \Rightarrow\ \rs=0$$
by the CHSY Lemma (\cite[Corollary 3.2]{CHSY} or \cite[Corollary 8.87]{BPT}). Hence we can define a norm on $\cVsa_{2k}$ by $\|\rs\|_{\bullet}\coloneq \|\rs(X)\|$. Also, finite-dimensionality of $\cVsa_{2k}$ implies that every element of $\cS_{2k}$ can be written as a sum of $N=1+\dim\cVsa_{2k}$ hermitian squares by Carath\'eodory's theorem \cite[Theorem I.2.3]{Bar}. Assume that a sequence $\{\rr_n\}_n\subset \cS_{2k}$ converges to $\rs\in \cVsa_{2k}$ with respect to $\|\cdot\|_{\bullet}$. If
$$\rr_n=\sum_{j=1}^N \rs_{n,j}^*\rs_{n,j},\qquad \rs_{n,j}\in\cV_k,$$
then the definition of our norm implies $\|\rs_{n_j}\|^2\le \|\rr_n\|$. In particular, the sequences $\{\rs_{n,j}\}_n\subset \cV_k$ for $1\le j\le N$ are bounded. Hence, after restricting to subsequences, we may assume that they are convergent: $\rs_j=\lim_n\rs_{n,j}$ for $1\le j\le N$. Consequently we have
\[\rs=\lim_n\rr_n=\sum_{j=1}^N \lim_n\left(\rs_{n,j}^*\rs_{n,j}\right)=\sum_{j=1}^N\rs_j^*\rs_j\in\cS_{2k}. \qedhere\]
\end{proof}

\subsection{Moore-Penrose evaluations}

In this subsection we generalize our notion of an evaluation of a nc rational expression. For $A\in\mat{n}$ let $A^\dagger\in\mat{n}$ be its Moore-Penrose pseudoinverse \cite[Section 7.3]{HJ}. Its properties that will be used in this section are
$$(A^\dagger)^*=(A^*)^\dagger,\qquad A^*=A^\dagger A A^*.$$
Given $r\in\re$ and $X\in\all$ we recursively define the {\bf Moore-Penrose evaluation} of $r$ at $X$, denoted $r\mpe(X)$:
\begin{enumerate}[label=(\alph*)]
	\item $\alpha\mpe(X)=\alpha I$ for $\alpha\in\kk$;
	\item $x_j\mpe(X)=X_j$ for $1\le j\le g$;
	\item $(r_1+r_2)\mpe(X)=r_1\mpe(X)+r_2\mpe(X)$ and
	\\ $(r_1r_2)\mpe(X)=r_1\mpe(X)r_2\mpe(X)$ for $r_1,r_2\in\re$;
	\item $(r^{-1})\mpe(X)=(r\mpe(X))^\dagger$ for $r\in\re$.
\end{enumerate}

Loosely speaking, with this kind of evaluation we replace all the inverses in an expression with Moore-Penrose pseudoinverses and the evaluation is then defined at any matrix point. Moore-Penrose evaluations of nc rational expressions frequently appear in control theory; see e.g.~\cite{BEFB}. We warn the reader that in general these evaluations do not respect the equivalence relation defining nc rational functions; also, Moore-Penrose evaluation is defined even for degenerate expressions. For example, $(0^{-1})\mpe(X)=0$ for all $X\in \all$. However, if $r\in\re$ is non-degenerate and $X\in\dom r$, then
$$r\mpe(X)=r(X)=\rr(X).$$

\begin{prop}\label{p:gns}
Let $r$ be a non-degenerate expression and assume the notation from the beginning of the section. If $\lambda:\cVsa_{2k+2}\to\R$ is a $\R$-linear functional satisfying $\lambda(\cS_{2k+2}\setminus\{0\})\subseteq \R_{>0}$, then there exists a scalar product $\langle\cdot,\cdot\rangle$ on $\cV_k$ and self-adjoint operators $X_j$ on $\cV_k$ such that
$$\lambda(\qq)=\left\langle q\mpe(X)1,1\right\rangle$$
for every $q=q^*\in Q$ with $4\tau(q)\le k$.	
\end{prop}

\begin{proof}
Let $\Lambda:\cV_{2k+2}\to\kk$ be the $\kk$-linear functional given by $\Lambda(\rs)\coloneq\lambda(\real \rs)+i\lambda(\imag \rs)$ if $\kk=\C$ and $\Lambda(\rs)\coloneq \lambda(\real \rs)$ if $\kk=\R$. Now $\langle \rs_1,\rs_2\rangle\coloneq\Lambda(\rs_2^*\rs_1)$ defines a scalar product on $\cV_{k+1}$. Let $\pi:\cV_{k+1}\to\cV_k$ be the orthogonal projection. If $x_j\in Q$, then define 
$$X_j:\cV_k\to\cV_k,\qquad \rs\mapsto\pi(x_j\rs).$$
It is clear that $X_j$ is a self-adjoint operator. We claim the following:
\\
\\
$(\star)$ {\it if $q\in Q$, then $q\mpe(X)\rs=\qq\rs$ holds for $s\in \bigcup_{j=0}^k \overbrace{Q\cdots Q}^j$ satisfying $4\tau(q)+\tau(s)\le k$.}
\\
\\
We prove $(\star)$ by the induction on the construction of $q$. Firstly, $(\star)$ obviously holds for $q\in \kk$ or $q\in Q\cap\{x_1\dots,x_g\}$. Next, if $(\star)$ holds for $q_1,q_2\in Q$ such that $q_1+q_2\in Q$ or $q_1q_2\in Q$, then it also holds for the latter. Finally, suppose that $(\star)$ holds for $q\in Q$ and assume $q^{-1}\in Q$. If $s\in \bigcup_{j=0}^k Q\cdots Q$ and $4\tau(q^{-1})+\tau(s)\le k$, then $4\tau(q)+(\tau(q^{-*})+\tau(q^{-1})+\tau(s))\le k$ and $q^{-*}q^{-1}s\in \bigcup_{j=0}^k Q\cdots Q$, so
\begin{align}
q\mpe(X)^*(\qq^{-*}\qq^{-1}\rs)&=(q^*)\mpe(X)(\qq^{-*}\qq^{-1}\rs)=\qq^*\qq^{-*}\qq^{-1}\rs=\qq^{-1}\rs, \label{e:31}\\
q\mpe(X)(\qq^{-1}\rs)&=\qq\qq^{-1}\rs=\rs. \label{e:32}
\end{align}
Since $\qq^{-1}\rs$ lies in the image of $q\mpe(X)^*$ by \eqref{e:31}, we have
$$\qq^{-1}\rs=(q\mpe(X)^\dagger q\mpe(X))(\qq^{-1}\rs)=(q\mpe(X))^\dagger\rs=(q^{-1})\mpe(X)\rs$$
by \eqref{e:32}.

In particular, if $q=q^*$ satisfies $4\tau(q)\le k$, then
$$\left\langle q\mpe(X)1,1\right\rangle=\Lambda\left(q\mpe(X)1\right)=\lambda(\qq)$$
by $(\star)$.
\end{proof}

\begin{prop}\label{p:nonneg}
Let $r\in\re$ be a non-degenerate expression and $t=\tau(r)$. If $\rr\notin\cS_{8t+2}$, then there exists $X\in\allh$ of size $\dim\cV_{4t}$ such that $r\mpe(X)$ is not positive semidefinite.
\end{prop}

\begin{proof}
If $\rr\neq\rr^*$, then there clearly exists $X\in\allh$ such that $\rr(X)$ is not self-adjoint; hence we assume that $\rr=\rr^*$. By Proposition \ref{p:closed} and the Hahn-Banach separation theorem \cite[Theorem III.1.3]{Bar} there exists a $\R$-linear functional $\lambda:\cVsa_{8t+2}\to\R$ such that $\lambda(\cS_{8t+2}\setminus\{0\})\subseteq \R_{>0}$ and $\lambda(\rr)<0$. By Proposition \ref{p:gns} there exists $X\in\allh$ and a vector $\vv$ of size $\dim\cV_{4t}$ such that
\[\left\langle r\mpe(X)\vv,\vv\right\rangle=\lambda(\rr)<0. \qedhere\]
\end{proof}

\begin{rem}
The converse of Proposition \ref{p:nonneg} does not hold. For example, if $r=x^{-1}x-1$, then $\rr=0$ is a sum of hermitian squares, but $r\mpe(0)=-1$ is not positive semidefinite.
\end{rem}

\subsection{Regular positive rational functions}

As a consequence of Proposition \ref{p:nonneg} we obtain the following version of Artin's theorem. It is a rational function analog of Helton's sum of hermitian squares theorem \cite{Hel}.

\begin{thm}\label{t:nonneg}
Let $\rr\in\rx$ be regular. Then $\rr(X)\succeq0$ for all $X\in\allh$ if and only if
$$\rr=\sum_j\rs_j^*\rs_j$$
for some regular $\rs_j\in\rx$.
\end{thm}

\begin{proof}
The implication $(\Leftarrow)$ is clear, so consider $(\Rightarrow)$. Since $\rr$ is regular, it admits a regular rational expression $r$ by Corollary \ref{c:expr}. If $\rr$ is not a sum of hermitian squares, then there exists $X\in\allh$ such that $r\mpe(X)$ is not positive semidefinite by Proposition \ref{p:nonneg}. Since $X\in\domh r$, we have $r\mpe(X)=\rr(X)$.
\end{proof}

\begin{rem}
Let $\rr\in\rx$ be a regular rational function that is not a sum of hermitian squares. If $r\in\re$ is its regular representative and $t=\tau(r)$, then there exists $X\in\allh$ of size $\dim\cV_{2t}$ such that $r\mpe(X)$ is not positive semidefinite. Indeed, this improved bound follows by replacing 4 with 2 in Proposition \ref{p:gns} since \eqref{e:31} becomes unnecessary when dealing with proper inverses.
\end{rem}

For strictly positive regular nc rational functions also see Remark \ref{r:pos}. Moreover, by the same reasoning as in Subsection \ref{ss:xx}, a suitably modified version of Theorem \ref{t:nonneg} holds for nc rational functions in $\ulx$ and $\ulx^*$ over $\C$.

%%%%%%%%%%%%%%%%%%%%%%%%%%%%%%%%%%%%%%%%%%%%%
%%%%%%%%%%%%%%%%%%%%%%%%%%%%%%%%%%%%%%%%%%%%%

\section{Examples and algorithms}\label{sec5}

In this section we present efficient algorithms to check 
whether $\rr\in\rx$ is regular and whether it is a sum of hermitian based on semidefinite programming. We finish the section with worked out examples.

\subsection{Testing regularity}

Our main results are effective
and enable us to devise an algorithm to check for 
regularity of a nc rational function.

\subsubsection{\Privileged pencils}\label{ssec:alg1}

Let $L=A_0+\sum_{j>0}A_jx_j\in \mat{d\times e}\otimes \px$
with $d\geq e$. (If $d<e$, simply replace $L$ by $L^*$.)
Now solve the following feasibility semidefinite program (SDP) 
for $D\in \mat{e\times d}$:
\begin{equation}\label{eq:sdp1}
\begin{split}
\real(DA_0)& \succeq0  \\
\trc(DA_0) & =1 \\
\real(DA_j)&=0 \quad\text{for } j>0.
\end{split}
\end{equation}
 (See e.g.~\cite{WSV,BPT} for more on SDPs.)
If \eqref{eq:sdp1} is infeasible, then $L$ is not \privileged. If
the output is a $D$ with
$\real(DA_0)\succ0$, then $L$ is (\stably) \privileged. Otherwise replace
$L$ by the linear pencil $LV$, where the columns of $V$ form a basis for
$\ker \real(DA_0)$, and repeat the algorithm.
Since $LV$ is of smaller size than $L$, the procedure will eventually terminate.

\subsubsection{Regular nc rational functions}
Given $\rr\in\rx$, we 
use an efficient (linear-algebra-based) algorithm, cf.~\cite[Section II.3]{BR},
to construct a realization of $\rr$ and then reducing it to a minimal one, say
$\rr=\cc^* L^{-1}\bb,$
where $L$ is a  $d\times d$ pencil. Now use the algorithm in Subsection \ref{ssec:alg1} below
to check whether $L$ is \privileged. By Theorem \ref{t:privreal}, $\rr$ is 
regular if and only if $L$ is \privileged.

\begin{rem}\label{r:pos}
This algorithm also yields a procedure to check whether a regular rational function is strictly positive everywhere: for a regular $\rr\in\rx$ we have $\rr(X)\succ0$ for all $X\in\allh$ if and only if $\rr(0)\succ0$ and $\rr^{-1}$ is regular. In particular, this can be applied for testing whether a nc polynomial is positive everywhere.
\end{rem}

\subsection{Testing positivity}
In this subsection we present an efficient algorithm
based on SDP
 to check whether a regular rational function is a sum of hermitian squares, i.e., whether it is positive everywhere.
We point out this is in sharp contrast to the classical commutative case \cite{BCR} where 
no efficient algorithms exist (in $>2$ variables) to check whether a rational function
$r\in\R(X)$ is globally positive.

Let $\rr\in\rx$. A {\bf positively \privileged realization} is one of the form
$$\rr=\bb^* \left(A_0+\sum_jA_jx_j\right)^{-1}\bb,\qquad \real A_0\succeq0,\qquad \real A_j=0 \quad\text{for } j>0.$$

\begin{prop}\label{p:sos}
A nc rational function $\rr=\rr^*\in\rx$ is a sum of hermitian squares if and only if it admits a positively \privileged realization.
\end{prop}

\begin{proof}
$(\Rightarrow)$ Let $\rr=\sum_j\rs_j^*\rs_j$ and $\rs_j=\cc_j^*L_j^{-1}\bb_j$. Then
$$\rs_j^*\rs_j=
\begin{pmatrix}0 & b_j^*\end{pmatrix}
\begin{pmatrix}\cc_j\cc_j^* & L_j^* \\ -L_j & 0\end{pmatrix}^{-1}
\begin{pmatrix}0 \\ b_j\end{pmatrix}
$$
is a positively \privileged realization and the direct sum of these pencils and vectors yields a positively \privileged realization of $\rr$.  

$(\Leftarrow)$ Let $\rr=\bb^* L^{-1}\bb$ be a positively \privileged realization. Then
$$\real \left(L^{-1}\right)=L^{-*}\real (L)L^{-1}=L^{-*}R^*RL^{-1}$$
for some constant matrix $R$ by assumption. Since $\rr$ is self-adjoint we have
$$\rr=\real\left(\bb^* L^{-1}\bb\right)=\bb^*\real \left(L^{-1}\right)\bb
=\left(RL^{-1}\bb\right)^*\left(RL^{-1}\bb\right),$$
so $\rr$ is a sum of hermitian squares.
\end{proof}

\begin{rem}
If $\kk=\R$, then a positively \privileged realization automatically yields a symmetric nc rational function.
\end{rem}

\subsubsection{Sum of squares testing}

Retain the notation of Section \ref{sec4}. 
Let $t=\tau(r)$ and suppose that $\rr=\rr^*$. To test whether $\rr$ is a sum of
hermitian squares, we proceed as follows. Pick a basis for $\cV_{2t+1}$ and stack
the elements of the basis into a vector, say $W$. Then $\rr$ is a sum of squares
if and only if the SDP \eqref{eq:sos} is feasible.
\begin{equation}\label{eq:sos}
\begin{split}
G & \succeq0\\
\text{s.t. \;}\rr&=W^* G W.
\end{split}
\end{equation}
To implement the equality in \eqref{eq:sos} 
we evaluate $\rr$ and $W^*GW$
at sufficiently many tuples of random self-adjoint matrices $X\in\allh$ of size
\begin{equation}\label{e:51}
\kappa(r)(1+(2t+1)(\dim\cV_{2t+1})^2),
\end{equation}
where
$$\kappa(r)=\# (\text{\small constant terms in }r)
+2\cdot \# (\text{\small symbols in }r)
+\# (\text{\small inverses in }r).$$ 
We refer to \cite[Subsection 6.1]{Vol} for the bound \eqref{e:51}.

Each solution $G$ to \eqref{eq:sos} yields a sum of squares decomposition of $\rr$. Namely,
letting $G=H^*H$ and $\rs=HW$, we have $\rr=\rs^*\rs = \sum_j \rs_j^* \rs_j$, 
where $\rs_j$ are the entries of the vector $\rs$. Finally, as in the proof of Proposition 
\ref{p:sos}, such a sum of squares decomposition can be employed to construct a positively \privileged realization for $\rr$.

\begin{exa}
A sum of hermitian squares does not necessarily admit a positively \privileged realization that is also a minimal one. For example, $x_1$ has a realization of size $2$, so $x_1^2$ has a positively \privileged realization of size $4$ by the proof of Proposition \ref{p:sos}. However, it can be checked that $x_1^2$ admits a realization of size $3$ but does not admit a positively \privileged realization of size $3$.
\end{exa}

\subsection{Examples}\label{ss:exa}

\begin{exa}
Let $\kk=\R$ and
$$\rr=\left((1-x_1x_2)(1-x_2x_1)+x_1^2\right)^{-1}.$$
While one can show that $\rr$ is regular using elementary arguments, we demonstrate this fact by applying our algorithm. Firstly we take a minimal realization of $\rr$ with the corresponding pencil
$$L=
\begin{pmatrix}
1 & 0 & 0 & -x_2 \\
0 & 1 & -x_2 & 0 \\
0 & -x_1 & 1 & -x_1 \\
-x_1 & 0 & x_1 & 1
\end{pmatrix}
=A_0+A_1x_1+A_2x_2.
$$
The system $\real(DA_1)=\real(DA_2)=0$ has a solution space of dimension $2$. Adding the constraint $\real(DA_0)\succeq0$ we obtain a one-dimensional salient convex cone $\cC$. For every $D\in\cC$ we have $\rk\real(DA_0)=2$, so $\rr$ is not \stably bounded. By choosing
$$D=
\begin{pmatrix}
0 & 0 & -1 & 0 \\
0 & 0 & 0 & 1 \\
1 & 0 & 1 & 0 \\
0 & -1 & 0 & 1
\end{pmatrix}
$$
we get
$$\real(DA_1)=\real(DA_2)=0,\qquad \real(DA_0)=\diag(0,0,1,1).$$
Hence let
$$L'=
\begin{pmatrix}
1 & 0 \\
0 & 1 \\
0 & -x_1 \\
-x_1 & 0
\end{pmatrix}.
$$
By choosing
$$D'=\begin{pmatrix}1 & 0 & 0 & 0\\ 0 & 1 & 0 & 0\end{pmatrix}$$
we verify that $L'$ is \privileged. Therefore $L$ is \privileged and so $\rr$ is regular.
\end{exa}

\begin{exa}
The rational function
$$\rr=\left(
2+ (x_1x_2-x_1-2x_2) \left(1+x_2^2\right)^{-1}+ (x_1+x_2-1) \left(1+x_2^2\right)^{-1}x_1
\right)^{-1}$$
is \stably bounded since it admits a realization with the \stably \privileged pencil
$$\begin{pmatrix}
1 & -x_2 & 0 \\
x_2 & 1 & -x_1-x_2+1 \\
0 & x_1+x_2-1 & 1 
\end{pmatrix}.$$
\end{exa}

\begin{exa}
Let $\kk=\C$ and consider
$$\rr=\left(1 + x_2^2 - ((1-i) x_1 + x_2) 
\left(1 + 2 x_1^2\right)^{-1} ((1+ i) x_1 + x_2)\right)^{-1}.$$
Note that $(X_1,X_2)\in\domh\rr$ for every pair of commuting hermitian matrices $X_1$ and $X_2$. It can be checked that $\rr$ admits a minimal realization with the pencil
$$L=
\begin{pmatrix}
1 & (-1-i) x_1 & -x_1 & 0 \\
0 & 1 & 0 & -x_1 \\
(-1+i) x_2 & 0 & 1 & -x_2 \\
0 & (-1-i) x_2 & 0 & 1
\end{pmatrix}
=A_0+A_1x_1+A_2x_2.$$
If
$$D=\begin{pmatrix}
0 & 0 & 0 & i \\
0 & 1+i & i & 0 \\
0 & i & 0 & 0 \\
i & 0 & 0 & \frac{1}{2}+\frac{i}{2}
\end{pmatrix},
$$
then
$$\real(DA_1)=\real(DA_2)=0,\qquad \real(DA_0)=\diag\left(0,1,0,\frac12\right).$$
Therefore we are left with
$$L'=\begin{pmatrix}
1 & -x_1 \\
0 & 0 \\
(-1+i) x_2 & 1 \\
0 & 0 
\end{pmatrix}=A_0'+A_1'x_1+A_2'x_2.$$
But for every $D'$ satisfying $\real(D'A_1')=\real(D'A_2')=0$ we have $\real(D'A_0')=\begin{psmallmatrix}
0&\alpha \\ \bar{\alpha}& 0\end{psmallmatrix}$ for $\alpha\in\C$, so $L'$ and $L$ are not \privileged. Hence $\rr$ is not regular. To find $X\in \mathcal{M}^2_{\operatorname{sa}}\setminus\dom \rr$ consider the structure of $L'$. We see that if $\frac12+\frac{i}{2}$ is an eigenvalue of $X_2X_1$, then $(X_1,X_2)\notin\domh\rr$. For a concrete example, take
$$X_1=\begin{pmatrix}
0 & 1+i \\
1-i & 0
\end{pmatrix},\qquad
X_2=\begin{pmatrix}
0 & \frac{i}{2} \\
-\frac{i}{2} & 0
\end{pmatrix}.
$$
\end{exa}

\begin{exa}
Again let $\kk=\R$. Another nontrivial example is
the following inverse of a sum of hermitian squares:
$$\rr=\left(\left(1+(x_2x_1)^2x_2^2\right)\left(1+x_2^2(x_1x_2)^2\right)-\left(x_1x_2-x_2x_1\right)^2\right)^{-1}.$$
The size of its minimal realization is $15$. With routine computation one observes that
$$\real(DA_1)=\real(DA_2)=0 \text{ and } \real(DA_0)\succeq0 \ \Rightarrow \ \real(DA_0)=0,$$
so $\rr$ is not regular. However, it is not apparent which concrete tuple of symmetric matrices is not contained in $\domh\rr$; using brute force one can check that $\domh\rr=\opm_2^{\operatorname{sa}}(\R)$ and $\domh\rr\neq\opm_3^{\operatorname{sa}}(\R)$.
\end{exa}

%%%%%%%%%%%%%%%%%%%%%%%%%%%%%%%%%%%%%%%%%%%%%
%%%%%%%%%%%%%%%%%%%%%%%%%%%%%%%%%%%%%%%%%%%%%

\end{document}